\documentclass[12pt]{amsart}
\usepackage{amsfonts}
\usepackage{latexsym,amsmath,amsthm,amssymb,amsfonts,mathrsfs}
\usepackage{graphicx}
\usepackage[usenames, dvipsnames]{color}
\usepackage{tikz}
\usepackage{float}
\usepackage{ulem}

\numberwithin{equation}{section}
\newtheorem{thm}{Theorem}[section]
\newtheorem{cor}[thm]{Corollary}

\newtheorem{lem}[thm]{Lemma}
\newtheorem{prop}[thm]{Proposition}
\newtheorem{dfn}[thm]{Definition}
\theoremstyle{remark}

\begin{document}

\title[Compactness of HSL]{Compactness of Hamiltonian Stationary Lagrangian submanifolds in symplectic manifold}
\author{Jingyi Chen, John Man Shun Ma}

\address{Jingyi Chen \\
Department of Mathematics \\ 
The University of British Columbia \\
1984 Mathematics Rd, Vancouver, BC V6T 1Z2, Canada} 
\email{jychen@math.ubc.edu.ca}

\address{John Man Shun Ma \\ 
Department of Mathematical Sciences \\
University of Copenhagen\\
Universitetsparken 5\\
DK-2100 Copenhagen Ø, Denmark}
\email{jm@math.ku.dk}

\begin{abstract}
In this work, we prove a compactness theorem on the space of all Hamiltonian stationay Lagrangian submanifolds in a compact symplectic manifold with uniform bounds on area and total extrinsic curvature. This generalizes the compactness theorems in \cite{CW2} and \cite{CMa}.  
\end{abstract}

\date{}

\thanks{Chen is partially supported by an NSERC Discovery Grant (22R80062). Ma is partially supported by DFF Sapere Aude 7027-00110B, by CPH-GEOTOP-DNRF151 and by CF21-0680 from respectively the Independent Research Fund Denmark, the Danish National Research Foundation and the Carlsberg Foundation.}

\maketitle
\section{Introduction}
Let $(M, \omega, h, J)$ be a $2n$-dimensional symplectic manifold with a symplectic 2-form $\omega$, an almost complex structure $J$ and a compatible metric $h$. An immersed submanifold $L$ in $M$ is Lagrangian if $\dim L = n$ and $\omega|_L = 0$, and a Lagrangian immersion is Hamiltonian stationary (HSL) if it is a critical point of the volume functional among all Hamiltonian variations. The notion was first introduced in \cite{Oh1}, \cite{Oh2}. A compact and graded Lagrangian in a Calabi-Yau manifold is Hamiltonian stationary if and only if it is special \cite{TY}. Thus HSL submanifolds are natural generalization of special Lagrangians submanifolds for general symplectic manifolds. 

Examples of HSL submanifolds include the totally geodesic $\mathbb{RP}^n$ in $\mathbb{CP}^n$ and the flat tori
$\mathbb S^1(a_1) \times \cdots \times \mathbb S^1(a_n)$ in $\mathbb C^n$ \cite{Oh1}, \cite{Oh2}. In two dimension, Schoen and Wolfson constructed in \cite{SW} HSL surfaces with conical singularity in any K{\"a}hler surfaces, which are also area minimizer in its Lagrangian homology class. On the other hand, using perturbation methods, Joyce, Lee and Schoen proved in \cite{JLS} the existence of closed HSL submanifolds in every compact symplectic manifolds (see also the previous works \cite{BC,Lee} and a family version in \cite{LR}). A large classes of examples of HSLs are constructed using techniques in integrable system. We refer the readers to the bibliography of \cite{CMa} for a more comprehensive list of references.

%explicit examples in various K{\"a}hler ambient manifolds \cite{A, AC1, CU, CDVV, MS, Moriya}; a complete classification of HSL tori in $\mathbb C^2$ via techniques in integrable systems and in $\mathbb{CP}^2$ \cite{HR1} and other homogeneous K{\"a}hler surfaces \cite{HR2, HM, Ma, MR, Mironov}, 

 Using the regularity theory developed in \cite{CW1}, Chen and Warren obtain the smoothness estimates and small Willmore energy regularity in \cite{CW2} and prove a compactness theorem for HSL submanifolds in $\mathbb C^n$ with uniformly bounded areas and total extrinsic curvatures in $\mathbb C^n$. The regularity and compactness results in \cite{CW1, CW2} rely on the assumption that the ambient space is $\mathbb C^n$ since it is used, in an essential way, that the Lagrangian phase angle $\Theta$ can be written as $\arctan\lambda_1+\dots +\arctan\lambda_n$ for the graphic representation $(x,Du)$, where $\lambda_i$'s are the eigenvalues of $D^2u$. 

In another direction, using the bubble tree convergence for conformal mappings with uniformly bounded area and Willmore energy, we proved in \cite{CMa} a compactness theorem on the space of compact HSL surfaces with bounded area, genus and Willmore energy in a K{\"a}hler surfaces. A sequence of such HSL surfaces would converges smoothly away from finitely many points to a branched HSL immersion. Unlike the case for $\mathbb C^n$ where one make use of the Lagrangian angle, in \cite{CMa} we exploit the fact that the mean curvature one form of a HSL immersion in a K{\"a}hler manifold satisfies a first order elliptic system of "Hodge type"; this is also essential in the regularity results in \cite{SW} for a Lagrangian area minimizers in two dimensions. 

\vspace{.2cm}

Our main result is 

\begin{thm} \label{main thm}
Let $(M, \omega, h)$ be a compact $2n$-dimensional symplectic manifold, where $h$ is a Hermitian metric compatible with $\omega$. Let $C_V$, $C_A$ be positive constants and let $(L_k)_{k=1}^\infty$ be a sequence of connected compact Hamiltonian stationary Lagrangian immersion in $M$ so that 
\begin{align}  \label{eqn bounded area and norm of A}
\operatorname{Vol}(L_k) \le C_V, \ \  \ \| A_k\|_{L^n} \le C_A, \ \ \forall k\in \mathbb N.
\end{align} 
Then either $(L_k)$ converges to a point, or there is a finite set $S$ so that a subsequence $(L_{k_i})$ of $(L_k)$ converges smoothly locally graphically to a Hamiltonian stationary Lagrangian immersion $L$ on $M\setminus S$ and 
\begin{equation} \label{volume identity}
\operatorname{Vol} (L) =\lim_{i\to \infty} \operatorname{Vol}(L_{k_i}).
\end{equation}
Also, the closure $\overline L$ in $M$ is connected and admits a structure of a Lagrangian varifold, and is Hamiltonian stationary in the sense that the generalized mean curvature vector $\vec H$ of the varifold $(\overline L, \mu_L)$ satisfies 
\begin{align} \label{mu_L satisfies HSL}
\int_M h (J\nabla ^M f, \vec H) d\mu_L = 0,  \ \ \forall f\in C^\infty_c(M).
\end{align}
\end{thm}

This generalizes the compactness theorems in \cite{CW2}, \cite{CMa} to HSLs in any compact symplectic manifolds with metrics compacitble with the symplectic structures and the almost complex structures. The main step in proving compactness is to derive a local a-priori estimates for HSLs with small total curvature. In \cite{CW2}, \cite{CMa}, this is done by representating the Hamiltonian stationary conditions as a coupled elliptic systems of lower order. This simplification is not available for HSLs in a general symplectic manifold. Therefore, the methods in \cite{CW2}, \cite{CMa} do not directly yield a proof of Theorem \ref{main thm}.  In a Darboux coordinates, if one represents the HSL as a gradient graph $(x, Du(x))$ of a function $u$, then the Hamiltonian stationary condition is captured by a fourth order PDE on $u$. In a recent work \cite{BCW}, Bhattachaya, Chen and Warren have studied a fourth order elliptic PDE of double divergence form and derive a regularity results. Together with the constructions of Darboux coordinates with estimates in \cite{JLS} (See also subsecion \ref{subsection darboux with estimates}), it is proved that a $C^1$ HSL in a general symplectic manifolds is smooth. In this work, we apply the results in \cite{BCW} to derive a-priori estimates (Theorem \ref{Theorem higher order estimates for u satisfying HSL}) and a $\epsilon$-regularity result (Theorem \ref{thm epsilon regularity}) for HSL submanifolds, which is essential to the proof of Theorem \ref{main thm}.

%In light of the two-dimensional structure of the variation problem of the area functional,  a strong compactness theorem \cite[Proposition 4.7]{SW}, among other important results, is proved  for weakly conformal, minimizing Lagrangian maps with a uniform area bound, i.e. a subsequence converges in the $W^{1,2}_{loc}$-topology to a minimizing Lagrangian map; this is applied to develop a deep theory of existence and regularity for minimizing Lagrangian maps \cite{SW}.  We employ the bubble tree convergence of conformal mappings that parametrize the HSLs and use  the construction in \cite{CLi}, while the bubble tree convergence for harmonic maps is first constructed in \cite{P} since the seminal work \cite{SU}. Theorem \ref{main thm} describes the singular points in the limit as branch points, and excludes the conical singularities in \cite{SW} since they have infinite Willmore energy. % (there is no non-flat special Lagrangian cone in two dimension).   
%Without a uniform bound on the Willmore energies, Theorem \ref{main thm} fails: the sequence of HSL tori 
%$\{ \mathbb S^1(1) \times \mathbb S^1(1/n)\}$ in $\mathbb C^2$ 
%has uniform bound on areas but not on the Willmore energies, and the limit is not a branched immersion.

\vspace{.2cm}

The paper is organized as follows. In section \ref{Background}, we discuss some background in Lagrangian submanifolds, which includes the definition of HSL and the Darboux coordinates with estimates constructed in \cite{JLS}. In section \ref{Estimates for double divergence equation}, we use the regularity results in \cite{BCW} to derive local $C^k$ estimates for functions which satisfy a fourth order elliptic equation in double divergence form. In section \ref{Local calculations (Symplectic)}, we prove a $\epsilon$-regularity theorem. Theorem \ref{main thm} is proved in sections \ref{pf of main thm n=1} and \ref{pf of main thm n>1}.

In the following, given any vector space $W$ with a metric $h$, we denote $|w|_{h} = \sqrt{h(w, w)}$ for $w\in W$. When $W$ is the Euclidean space, we write $|w| = |w|_{h_0}$, where $h_0$ is the Euclidean metric on $W$. Given any open subset $\Omega$ of $\mathbb R^n$ and any function $u : \Omega \to \mathbb R$ on $\Omega$, we use $\| u\|_{C^{k, \alpha} (\Omega)}$, $\| u\|_{L^p (\Omega)}$ and $\|u\|_{W^{k,p}(\Omega)}$ to denote respectively the $C^{k, \alpha}$, $L^p$ and $W^{k,p}$ norms of $u$. For any $k\in \mathbb N$, $x\in \mathbb R^k$ and $r>0$, the open ball with radius $r$ and center $x$ in $\mathbb R^k$ is denoted $B^k_r(x)$. We also write $B^k_r$ for $B^k_r (0)$.

\section{Background} \label{Background}
\subsection{Hamiltonian stationary Lagrangian submanifolds} Let $(M, h)$ be a Riemannian manifold. Unless otherwise specified, throughout this paper a submanifold of $M$ is an immersed submanifold, that is, a submanifold of $M$ is a pair $(L, \iota)$, where $L$ is a subset of $M$ and $\iota : N\to M$ is an immersion with $\iota(N)= L$. An immersed submanifold is compact (resp. connected), if $N$ is compact (resp. connected), and is proper if $\iota$ is a proper map.

Assume that $M$ is compact. By the Nash embedding theorem, there is an isometric embedding of $(M, h)$ into $\mathbb R^{K}$ for some positive integer $K$. Fixing such an embedding, the definition of varifolds on $M$ is given in \cite{Allard}. Let $U$ be an open subset of $M$, and let $L$ be a properly immersed $k$-dimensional submanifold of $U$ given by the immersion $\iota : N\to U$. Let $g = \iota^*h$ be the induced metric on $N$ and let $dV_g$ be the volume form. The volume of a submanifold $L$ is
\begin{equation} \label{eqn dfn of volume of L}
\operatorname{Vol}(L) = \int_N dV_g.
\end{equation}
Let $\pi_k: G_kM\to M$ be the Grassmann bundle on $M$, where each fiber at $x\in M$ is the Grassmann manifold of $k$-dimensional subspaces of $T_xM$. Given any $k$-dimensional immersed submanifold $L$ of $M$, let $G_\iota : N \to G_kM$ be the Gauss map given by $x\mapsto (\iota_*)_x T_xN$. Then $(L, \iota)$ is given a structure of a varifold $\mu_L$ by pushing forward: $|L| :=(G_\iota)_\sharp dV_g$, or 
\begin{equation} \label{eqn dfn of |L|}
    |L|(f) := \int_N f\circ G_\iota \ dV_g, \ \ \forall f\in C(G_kM).
\end{equation}
For any integral varifold $V$, the weight measure $\mu_V$ on $M$ is given by 
\begin{equation}  \label{eqn dfn of mu_L}
\int_M \phi d\mu_V := \int_{G_kM} \phi(x) dV, \ \ \forall \phi \in C(M). 
\end{equation}
Let $\mu_L:=mu_{|L|}$. Clearly $\mu_L$ and $|L|$ depend not only on $L$ but also on $\iota$. 

Let $(M, \omega, h, J)$ be a $2n$-dimensional symplectic manifold, where $\omega$ is the symplectic form, $J$ is an almost complex structure and $h$ is a Riemannian metric. We assume that $\omega, J, h$ satisfy the compatibility conditions
\begin{equation}\label{omega, h, J compatibility condition}
    h(X, Y) = \omega (X, JY), \ \ \ h(JX, JY) = h(X, Y) 
\end{equation}
for all tangent vectors $X, Y$ at the same point. 

An $n$-dimensional immersed submanifold $L$ in $M$ given by an immersion $\iota$ is called Lagrangian if $\iota^*\omega = 0$. Given an immersed Lagrangian submanifold $L$, (\ref{omega, h, J compatibility condition}) implies that $J$ maps tangent vectors of $L$ to normal vectors. Thus the second fundamental form of $L$ can be written as 
\begin{align} \label{formula for A}
A : T_pL  \times T_pL \times T_pL \to \mathbb R, \ \ \ A(X, Y, Z) = h(\overline \nabla_{X} \widetilde Y, JZ)= \omega (Z, \overline \nabla _X \widetilde Y),
\end{align}
where $\widetilde Y$ is any local extension of $Y$. The {\sl mean curvature one form} of $L$ is a one form on $N$ defined by 
\begin{equation} \label{eqn dfn mean curvature one form}
    \alpha_x (X) = \omega_{\iota(x)} (\iota_* X , \vec H(x))
\end{equation}
for any $x\in N$ and $X\in T_xN$. Here $\vec H$ is the mean curvature vector of the immersion $\iota : N \to M$.

Next we give the definition of a Hamiltonian stationary Lagrangian immersion. As in \cite{CW2}, we first give the definition for general Lagrangian integral $n$-varifolds (An integral $n$-varifold $V$ on $(M, \omega, h)$ is called {\sl Lagrangian} if $\omega|_{T_xV} = 0$ for $V$-almost every $x$). 

\begin{dfn} \label{dfn HSL for varifolds}
A Lagrangian integral $n$-varifold $V$ in $U$ is called Hamiltonian stationary Lagrangian (HSL) in $U$ if 
\begin{align} \label{dfn of HSL for varifold}
\int_{U} h (J\nabla ^M f, \vec H) d\mu_V = 0,  \ \ \forall f\in C^\infty_c(U),
\end{align} 
where $\vec H$ is the generalized mean curvature vector of $V$. 
\end{dfn}

\begin{dfn} \label{dfn HSL for immersion}
Let $L$ be a Lagrangian immersion into an open subset $U$ of $M$ given by a proper immersion $\iota$, and let $|L|$ be the integral $n$-varifold associated with $L$ as in (\ref{eqn dfn of |L|}). $L$ is called a Hamiltonian stationary Lagrangian (HSL) immersion in $U$ if $|L|$ is HSL as in Definition \ref{dfn HSL for varifolds}.
\end{dfn}

Let $\iota : N\to U$ be a proper Lagrangian immersion. For any $f\in C^\infty_c(U)$, let $\bar f =  f\circ \iota$. By (\ref{eqn dfn of mu_L}), (\ref{omega, h, J compatibility condition}) and (\ref{eqn dfn mean curvature one form}), if the immersed submanifold $(L, \iota)$ is HSL in $U$, then
$$ \int_N (d\bar f, \alpha)_g dV_g = 0,$$
where $\alpha$ is the mean curvature one form of $L$ defined in (\ref{eqn dfn mean curvature one form}). Let $L$ be an embedded Lagrangian submanifold on $M$. Since every smooth function with compact support $\psi$ on $L$ can be extended to $f \in C^\infty_c(M)$, $L$ is HSL if and only if 
$$\int_N (d\psi,\alpha)_g  dV_g = 0, \ \ \ \forall \psi \in C^\infty_c(N).$$
This implies that $\alpha$ is co-closed. That is,
\begin{equation} \label{eqn HSL d^* alpha = 0}
 d^* \alpha = 0, 
\end{equation}
where $d^* = \pm *d*$ is the formal adjoint of the exterior derivative $d$ with respect to the metric $g$. 

Let $L$ be an immersed submanifold in $M$ given by an immersion $\iota : N \to M$ and let $U\subset M$ be an open set in $M$. Decompose $\iota^{-1} (U)$ as connected components 
\begin{align}
\iota^{-1}(U) = \bigsqcup_i V_i.
\end{align}
If $\iota|_{V_i}$ is an embedding for each $i$, we say that each $L_i = \iota (V_i)$ is an {\sl embedded connected component} of $U\cap L$ and that $\iota$ splits into embedded connected components on $U$. 

As in \cite[Proposition 2.2]{CW2}, if $L$ is a proper immersion, for each $p \in L$ there is an open neighborhood $U$ of $p$ so that $L\cap U$ splits into finitely many embedded connected components. We will have a more precise statement in Proposition \ref{prop split into embedded components}.

%\begin{dfn}
%Let $L$ be an immersed Lagrangian in $(M, \omega, h)$ given by an immersion $\iota$. $L$ is called {\sl Hamiltonian stationary Lagrangian} at $p\in L$ if there is an open neighborhood $U$ of $p$ such that $\iota^{-1}(U)$ splits into embedded connected components, and for each such component $L^i$ we have 
%\begin{align} \label{dfn of HSL for immersion}
%\int_{L^i} h (J\nabla^M f, \vec H) d\mu_{L^i} = 0,  \ \ \forall f\in C^\infty_c(U).
%\end{align} 
%\end{dfn}

\begin{prop} \label{prop L is HSL iff alpha is co-closed}
Let $L$ be a Lagrangian submanifold in an open subset $U$ of $M$ given by a proper immersion $\iota : N \to U$. Then $L$ is HSL as in Definition \ref{dfn HSL for immersion} if and only if its mean curvature one form $\alpha$ is co-closed. 
\end{prop}

\begin{proof}
The proof is similar to that of \cite[Proposition 2.5 (1)]{CW2} and we only sketch the proof for the direction $(\Rightarrow)$. Let $p\in L$. By \cite[Proposition 2.2]{CW2}, let $W$ be an open neighborhood of $p$ in $U$ so that $\iota^{-1} (W)$ splits into finitely many connected components $\iota|_{E_i} : E_i \to \Sigma_i$, where $\Sigma_i = \iota (E_i)$. We recall that $q\in W\cap L$ is called an {\sl embedded point} of $L$ if there is an open neighborhood $V_q$ of $q$ in $W$ so that $L\cap V_q$ is an embedded submanifold. Let $\mathscr E\subset L\cap W$ be the set of all embedded points of $L$. Using embeddedness, one can argue as in above that $d^*\alpha = 0$ on $\iota^{-1} (\mathscr E)$ using Definition \ref{dfn HSL for immersion}. For each $i$, one can prove that $\mathscr E \cap \Sigma_i$ is dense in $\Sigma_i$. Hence $d^*\alpha = 0$ on $\iota^{-1}(W)$. Since $p\in L$ is arbitrary, $\alpha$ is co-closed. 
\end{proof}

In particular, the Hamiltonian stationary condition holds on each embedded connected components.  
\begin{cor} \label{cor HSL on each embedded connected component}
Let $L$ be a properly immersed HSL in an open subset $U$ of $M$. Let $W$ be an open subset of $U$ such that $W\cap L$ splits into embedded connected components $\iota|_{E_i} : E_i \to \Sigma_i$. Then for each $i$, 
\begin{equation} \label{eqn HSL condition on each embedded connected components}
    \int_{W} h(J \nabla^M f,  \vec H_i) d\mu_{\Sigma_i} = 0, \ \ \ \forall f\in C^\infty_c(W).  
\end{equation}
Here $\vec H_i$ is the mean curvature vector of $\Sigma_i$. 
\end{cor}

\subsection{Darboux coordinates with estimates} \label{subsection darboux with estimates}
In this subsection we recall the Darboux coordinates with estimates constructed in \cite{JLS}. Let $(M, \omega)$ be a compact symplectic $2n$-manifold, and let $h$ be a Riemannian metric on $M$ compatible with $\omega$. Let $\mathbb U$ be the $U(n)$ frame bundle of $M$: that is, for each $p\in M$, the fiber $\mathbb U_p$ consists of unitary linear mappings $v: \mathbb R^{2n} \to T_pM$. 

We recall Proposition 3.2 in \cite{JLS}: 
\begin{prop} \label{D coord. with estimates}
For small $\epsilon >0$, we can choose a family of embeddings $\Upsilon_{p,v} : B^{2n}_\epsilon \to M$ depending smoothly on $(p,v) \in \mathbb U$ such that for all $(p, v)\in \mathbb U$, we have 
\begin{itemize}
\item [(i)] $\Upsilon_{p,v}(0)= p$ and $d\Upsilon_{p,v} |_0 = v : \mathbb R^{2n} \to T_pM$;
\item [(ii)] $\Upsilon_{p, v\circ \gamma} \cong \Upsilon_{p,v} \circ \gamma$ for all $\gamma \in U(n)$;
\item [(iii)] $\Upsilon^*_{p,v} (\omega) = \omega_0 = \frac{i}{2} \sum_{j=1}^n dz_j \wedge d\bar z_j$; and 
\item [(iv)] $\Upsilon^*_{p,v} (h) = h_0 + O(|z|) = \sum_{j=1}^n |dz_j|^2 + O(|z|)$. 
\end{itemize}
\end{prop}
From now on, by scaling $\omega, h$ if necessary, we assume $\epsilon=2$ in Proposition \ref{D coord. with estimates}. 
\begin{dfn}
Let $p\in M$ and $0<r\le 1$. The {\sl Darboux ball} at $p$ with radius $r$ is 
\begin{align}
\mathscr B_r (p) := \Upsilon_{p,v} (B^{2n}_r).
\end{align}
\end{dfn}

By (ii) in Proposition \ref{D coord. with estimates}, $\mathscr B_r(p)$ is well-defined, independent of $v \in \mathbb U_p$.  For each $R\ge 1$ and $0<t\le R^{-1}$, we define $t : B^{2n}_R \to B^{2n}_1$ as $z\mapsto tz$. The following is proved in \cite[Proposition 3.4]{JLS}. 

\begin{prop} \label{D coord. with estimates scaled}
There exists $K_0, K_1, K_2, \cdots, $ depending only on $(M, \omega, h)$ such that for all $(p, v)\in U$ and $t\in (0,R^{-1}]$, the metric $ h^t_{p,v}  = t^{-2} (\Upsilon_{p,v} \circ t)^* h$ satisfies on $B^{2n}_R$
\begin{align} \label{eqn estimates on g^t}
| h^t_{p,v} - h_0| \le K_0 t \ \text{ and }  \ |\partial_j h^t_{p,v}|\le K_j t^j \ \text{ for }j=1,2,\cdots.
\end{align}
\end{prop}

Again, by scaling $\omega, h$ if necessary, we assume that 
\begin{align} \label{K_0, K_1 le epsilon_1}
\max\{ K_0, K_1\} \le \epsilon_1<\max\{1/2, \epsilon_1(1/2, n)\},
\end{align}
where $\epsilon_1(1/2, n)$ is given in \cite[Lemma 3.2]{BCW}. By the first inequality in (\ref{eqn estimates on g^t}) and $K_0 < 1/2$, 
\begin{align} \label{comparing geodesic ball and darboux ball}
B^M_{r/2} (p) \subset \mathscr B_r (p) \subset B^M_{2r} (p) 
\end{align}
for all $p\in M$ and $r \in (0,1)$, where $B^M_{s} (p)$ is the metric ball with radius $s$ in $(M, h)$. 

\subsection{Convergence of immersed submanifolds}
We recall the definition of smooth convergence as local graphs of a sequence of immersed submanifolds introduced in \cite{CW2} (see also \cite{CS}, \cite{Sharp} for similar definitions for sequences of embedded hypersurfaces). 

\begin{dfn} \label{dfn converges graphically in U}
Let $(\Sigma_k)_{k=1}^\infty$ be a sequence of properly immersed submanifolds in an open subset $U$ of a smooth manifold $M$ given by a proper immersion $\iota_k : N_k \to U$ for each $k$. Let $S$ be an integral varifold in $(M, h)$ and $m\in \mathbb N\cup\{\infty\}$. We say that $(\Sigma_k)_{k=1}^\infty$ converges graphically in $U$ to $S$ in $C^m$-topology, if there is $I \in\mathbb N \cup \{0\}$ such that 
\begin{itemize}
\item $\Sigma_k$ splits into $I$ embedded connected components $\iota_k|_{E^i_k} :E^i_k \to \Sigma^i_k$, $i=1, \cdots I$. Here $E_k^i$ is open in $N_k$ and $\iota_k|_{E^i_k}$ is an embedding onto $\Sigma^i_k$, and
\item $S \cap U = \Sigma^1 \cup \cdots \cup \Sigma^{I}$, where each $\Sigma^i$, $i=1, \cdots, I$, is an embedded submanifold in $U$.
\end{itemize}
 Moreover, for each $i = 1, \cdots, I$, there is $p_i \in \Sigma^i$ such that $\Sigma^i_k$, $k\in \mathbb N$, and $\Sigma^i$ are graphs of a smooth mappings $X^i_k$ and $X^i$ respectively, where $X^i_k$ and $X^i$ are defined on some open subsets of $T_{p_i} \Sigma^i$, such that $X^i_k$ converges in $C^m$ norm to $X^i$ as $k\to \infty$.
\end{dfn}
We remark when $I=0$, it is understood that $\Sigma_k = S\cap U =\emptyset$. 

\begin{dfn} \label{dfn local graphical smooth convergence}
Let $(\Sigma_k)_{k=1}^\infty$ be a sequence of properly immersed submanifolds in $M$, and let $V$ be an open subset of $M$. We say that $(\Sigma_k)_{k=1}^\infty$ converges in $V$ smoothly as local graphs to an integral varifold $S$, if for all $q\in V$, there is an open neighborhood $U$ of $q$ contained in $V$ such that $(\Sigma_k\cap U)_{k=1}^\infty$ converges graphically to $S\cap U$ in $C^\infty$-topology.  
\end{dfn}

\section{Estimates for double divergence equation} \label{Estimates for double divergence equation}
In \cite{BCW}, the authors study the following fourth order equation in double divergence form 
\begin{equation}\label{eqn double divergence equ classical form}
\partial_{x^i} \partial_{x^j} F^{jl} (x, Du, D^2u) = \partial_{x^k} a^k (x, Du, D^2u) - b(x, Du, D^2u)
\end{equation}
where $F^{ij}$, $a^k$ and $b$ are smooth functions in $x, Du, D^2u$, defined in a convex neighborhood $U\subset B^n_r \times \mathbb R^n \times S^{n\times n}$, where $S^{n\times n}$ is the space of $n\times n$ symmetric matrices. A $W^{2, \infty}$ function $u$ on $B^n_r$ is called a weak solution to (\ref{eqn double divergence equ classical form}) if
\begin{align} \label{double divergence form}
\int_{B_r} \big[ F^{ij} \eta_{ij} + a^k \eta_k + b\eta\big] dx = 0, \ \ \text{ for all } \eta\in C^\infty _c(B^n_r).
\end{align}

\begin{dfn} 
Let $\Lambda >0$. We say that (\ref{eqn double divergence equ classical form}) is $\Lambda$-uniform on $U$ if the standard Legendre ellipticity condition 
\begin{align} \label{Lambda uniform definition}
\frac{\partial F^{jl}}{\partial u_{ik} } (\xi ) \sigma_{ij} \sigma_{kl} \ge \Lambda |\sigma|^2 ,\ \forall \sigma \in S^{n\times n}
\end{align}
is satisfied for any $\xi \in U$. 
\end{dfn}
For each $p = 1, \cdots, n$, let $h_p = he_p$, 
\begin{align} \label{dfn of eta -hp}
\eta^{-h_p} (x) : = \frac{\eta(x-h_p) - \eta(x) }{h},
\end{align}
and $f:=u^{h_p}$. Define 
\begin{align*}
\xi_0 &= (x, Du(x), D^2u(x)), \\
\xi_h &= (x+ h_p, Du(x+h_p), D^2u(x+ h_p)), \\
\vec V &= \xi_h - \xi_0.
\end{align*}

One can check that (\ref{double divergence form}) implies 
\begin{align} \label{double divergence for f}
\int_{B_r} \left( \beta^{ij,kl} f_{ik} \eta _{jl} + \gamma^{il} \eta_{jl} + \psi^k \eta_k^{-h_p} + \zeta \eta^{-h_p}\right) dx = 0 \ \ \ \ \forall \eta \in C^\infty_c(B^n_{r-h}),
\end{align}
where 
\begin{align*}
\beta^{ij, kl} (x) &= \int_0^1 \frac{\partial F^{jl} }{\partial u_{ik} }(\xi_0 + tV) dt, \\
\gamma^{jl} (x) &= \int_0^1 \left( \frac{\partial F^{jl}}{\partial u_k }( \xi_0 +tV)f_k  + \frac{\partial F^{jl} }{\partial x_p} (\xi_0 + tV) \right)dt, \\
\psi^k(x)  &= a^k(x, Du, D^2u), \\
\zeta(x) &= b(x, Du, D^2u). 
\end{align*}

In the following, we assume that $u$ is at least $C^2$ and satisfies (\ref{double divergence form}). The following propositions are the key steps in proving the regularity results in \cite{BCW}.

\begin{prop} \label{prop C3alpha estimates}
Assume that $u \in C^{2,\alpha}(B^n_r)$ satisfies  (\ref{double divergence for f}), then for all $0<r'< r$ and $0<\alpha'< \alpha$ we have $u\in C^{3,\alpha' }(B^n_{r'})$ and 
\begin{align} \label{C 3alpha estimates in balls of arbitrary radius}
\|D^3 u\|_{C^{\alpha'} (B^n_{r'}) } \le C(\Lambda, \alpha, \alpha', r,r', \|u\|_{W^{2,\infty}(B^n_r)}, \|\Psi\|_{C^{\alpha} (B^n_r)})
\end{align}
where for simplicity we write 
\begin{align*}
\Psi = \big( (\beta_0 ^{ij,kl})_{i,j,k,l}, (\gamma_p^{jl})_{j,l,p}, (\psi^k)_k, \zeta\big). 
\end{align*}
and 
\begin{align}
\beta^{ij,kl}_0 (x) &= \frac{\partial F^{jl}}{\partial u_{ik}} (x,  Du, D^2u), \\
\label{dfn of gamma jl p} \gamma^{jl}_p (x) &=  \frac{\partial F^{jl}}{\partial u_{k}} (x,  Du, D^2u) u_{pk} +  \frac{\partial F^{jl}}{\partial x^p} (x,  Du, D^2u).
\end{align}
\end{prop}

\begin{proof}
We first assume $r' = 1/5$ and $r=1$. Without rewriting the whole proof again, we merely point out that in the proof of \cite[Proposition 2.3]{BCW}, the constants used have the following dependence: 
\begin{align*}
C_3 &= \|\beta^{ij,kl}\|_{C^\alpha (B^n_r)},\\
C_4 &= \|\beta^{ij,kl}\|_{C^\alpha (B^n_r)},\\
C_5 &= C(n) \| \gamma^{jl}\|_{C^\alpha (B^n_r)}, \\
C_6 &= C(n) \| \psi^k\|_{C^\alpha (B^n_r)}, \\
C_7 &= C(n) \|\zeta\|_{C^\alpha (B^n_r)}, \\
C_2 &=C(\Lambda)\\
C_1 &=C_1(n)  \\
\tilde \alpha &= 1-\delta, \\
\tilde q &= \frac{n}{2-2\tilde \alpha}, \\
\tilde K &= C(n) (\| \gamma^{jl} \|^2_{L^{2\tilde q} (B^n_1)} + \| \psi ^k\|^2_{L^{2\tilde q} (B^n_1)} + \|\zeta\|_{L^{2\tilde q} (B^n_1)}), \\ 
r_0 &= r_0(C_1, \Lambda, n, \delta), \\
C_8 &= C_8 (C_1, n, \delta), \\
C_9 &= C_9 (C_1, C_8, r_0, \Lambda, \|D^2 f\|_{L^2(B^n_{1/2})}) ,\\
C_{10} &= C_{10} (C_4, C_5, C_7, \Lambda, \alpha, \delta, r_0, \| D^2f\|^2_{L^2 (B^n_{1/2})}, C_2)
\end{align*}
Thus by Proposition 2.1 and Proposition 2.3 in \cite{BCW}, 
\begin{align} \label{C 3alpha estimats in B_1}
\| D^3 u \|_{C^{\alpha -\delta/2} (B_{1/5})} \le C(\Lambda, \alpha, \delta, \|u\|_{W^{2,\infty}(B^n_1)}, \|\Psi\|_{C^{\alpha} (B^n_1)}), 
\end{align}
In general, given $0<r'<r\le 1$ and $x_0 \in B^n_{r'}$, we scale the ball $B^n_{r-r'} (x_0)$ to radius $1$ and obtain, via (\ref{C 3alpha estimats in B_1}), 
\begin{align*}
\| D^3 u \|_{C^{\alpha -\delta/2} (B^n_{(r-r')/5} (x_0))} \le C(\Lambda, \alpha, \delta, r,r', \|u\|_{W^{2,\infty}(B^n_r)}, \|\Psi\|_{C^{\alpha} (B^n_r)}), 
\end{align*}
Since $x_0 \in B^n_{r'}$ is arbitrary, (\ref{C 3alpha estimates in balls of arbitrary radius}) is proved. 
\end{proof}

From the above proposition, we obtain
\begin{prop} \label{prop Ck alpha estimates}
For each $0<r'<r\le 1$, $0<\alpha' <\alpha<1$ and $k \ge 3$, assume that $u \in C^{2,\alpha} (B^n_{r})$ satisfies (\ref{double divergence form}) and $(F^{jl})$ satisfies the Legendre ellipticity condition (\ref{Lambda uniform definition}), then $u$ is smooth in $B_{r}$ and 
\begin{align} \label{C^k alpha estimates on u}
\|D^k u\|_{C^{\alpha'} (B^n_{r'}) } \le C(k, \Lambda,  \alpha, \alpha',r,r', \|u\|_{W^{2,\infty} (B^n_r)}, \|(\vec F, \vec a, b)\|_{C^{k-3} (B^n_r)}).
\end{align}
\end{prop}

\begin{proof}
That $u$ is smooth is proved in \cite{BCW}. Taking the limit $h\to 0$ in (\ref{double divergence for f}), we obtain (writing $f = u_p$)
\begin{align} \label{double divergence form for u_p}
\int_{B^n_{r}} \big[ (\beta^{ij,kl}_0 f_{ij} +\gamma^{jl}_p ) \eta_{jl}+ \partial_p \psi^k \eta_k +\partial_p\zeta  \eta \big] dx= 0, \ \ \forall \eta\in C^\infty_c(B^n_r).
\end{align}
This is an equation of the form (\ref{double divergence form}) with the leading term 
\begin{align} 
F^{ij}_{(p)} (x, Df, D^2f) = \beta^{ij,kl}_0 f_{ij} +\gamma^{jl}_p 
\end{align}
and in particular 
\begin{align}
\frac{\partial F^{jl}_{(p)} } {\partial f_{ik}} = \beta^{ij,kl}_0. 
\end{align}
Hence (\ref{double divergence form for u_p}) also satisfies the $\Lambda$-uniform condition (\ref{Lambda uniform definition}).

In general, one can inductively show that for any $p_1, \cdots, p_m$, the function $f = u_{p_1\cdots p_m}$ satisfies 
\begin{align}
\int_{B^n_r }\big[  (\beta^{ij,kl}_0 f_{ik} + \gamma^{jl}_{(p_1\cdots p_m)} ) \eta_{jl} + \psi_{p_1\cdots p_m}^k \eta_k + \zeta_{p_1\cdots p_m} \eta\big] dx = 0, \ \ \forall \eta\in C^\infty_c(B^n_r),
\end{align}
here $\gamma^{jl}_{(p_1\cdots p_m)} $ is defined inductively by 
\begin{align} \label{inductive dfn of gamma}
\gamma^{jl}_{(p_1\cdots p_m)} = \big( \gamma^{jl}_{(p_1\cdots p_{m-1})}\big)_{x_{p_m}}+ \big(\beta^{ij,kl}_0 \big)_{x_{p_m}} u_{ikp_1\cdots p_m}. 
\end{align}
By Proposition \ref{prop C3alpha estimates}, we obtain 
\begin{align} \label{C 3alpha estimates in balls of arbitrary radius for f=up}
\|D^3 f\|_{C^{\alpha'} (B^n_{r'}) } \le C(\Lambda, \alpha, \alpha', r,r', \|f\|_{W^{2,\infty}(B^n_r)}, \|\Psi_{(p_1\cdots p_m)}\|_{C^{\alpha} (B^n_r)}),
\end{align}
where
\begin{align}
\Psi_{(p_1\cdots p_m)} = \big( (\beta_0 ^{ij,kl})_{i,j,k,l}, (\gamma_{(p_1\cdots p_m)}^{jl})_{j,l,p_1,\cdots p_m}, (\partial_{p_1\cdots p_m}\psi^k)_k, \partial_{p_1\cdots p_m} \zeta\big).
\end{align}
Arguing inductively from (\ref{dfn of gamma jl p}), (\ref{inductive dfn of gamma}) that 
\begin{align*}
\| \Psi_{p_1\cdots p_m} \|_{C^\alpha (B^n_r)} \le C(m, \| \vec a\|_{C^m(B^n_r)}, \| F\|_{C^m(B^n_r)}, \|b\|_{C^m (B^n_r)}, \|u\|_{C^{m+2, \alpha}(B^n_r)})
\end{align*}
Hence we have for any $0 <r'<r<1$ and $0 < \alpha' < \alpha<1$, 
\begin{align} \label{inductive estimates on C m+3 alpha of u}
\|u\|_{C^{m+3, \alpha'} (B^n_{r'})} \le C(m,\Lambda, r, r', \alpha, \alpha',\|(\vec F, \vec a, b)\|_{C^m(B^n_r)}, \|u\|_{C^{m+2, \alpha} (B^n_{r})})
\end{align}
The proposition is then proved by choosing 
\begin{align*}
r' &= r_m < r_{m-1} < \cdots < r_1 < r_0 = r, \\
\alpha' &= \alpha_m < \alpha_{m-1} < \cdots < \alpha_1 < \alpha_0 = \alpha
\end{align*}
and applying (\ref{inductive estimates on C m+3 alpha of u}) $m-1$ times (with $r,r', \alpha, \alpha'$ replaced by $r_{j-1}, r_{j},\alpha_{j-1}, \alpha_j$ respectively) and lastly Proposition \ref{prop C3alpha estimates}.
\end{proof}

\section{Local calculations (Symplectic)} \label{Local calculations (Symplectic)}
In this section, we represent locally a Lagrangian immersion as a gradient graph in a Darboux coordinates chart given in Proposition \ref{D coord. with estimates}. We apply the estimates in the previous section. Under a smallness assumption on the $L^n$-norm of the second fundamental form, we derive a $\epsilon$-regularity result (Theorem \ref{thm epsilon regularity}). 

Let $h=h^t_{p,v}$ be a Riemannian metric on $B^{2n}_R$ which satisfies (\ref{eqn estimates on g^t}) for some $t\in (0,R^{-1}]$ and $K_0, K_1$ satisfy (\ref{K_0, K_1 le epsilon_1}). 

% \begin{align} \label{estimates on g local}
%| h - h_0| \le K_0 t \ \text{ and }  \ |\partial_j h|\le K_j t^j \ \text{ for }j=1,2,\cdots.
%\end{align}

%Let $J$ be a complex structure on $B^{2n}_R$ so that 
%\begin{align}\label{omega g, J compatibility}
% \omega_0 (X, Y) = h(JX, Y).
%\end{align}
 
Let $(x^1, \cdots x^{2n})$ be the standard coordinates on $B^{2n}_R$ and $e_1, \cdots, e_{2n}$ be the standard basis of $\mathbb R^{2n}$. The Christoffel symbols for the Levi-Civita connection of $h$ is given by 
 $$\overline\nabla _{e_i} e_j = \Gamma_{ij} ^k e_k + \Gamma_{ij}^{\bar k} e_{\bar k}, $$
where $\bar k := n+k$ and the repeated indices are summed from $1$ to $n$. 

 From (\ref{eqn estimates on g^t}) we have 
 \begin{align} \label{eqn estimates on Gamma} 
|\Gamma_{ij}^k| + |\Gamma_{ij}^{\bar k}| \le K_1 R^{-1}, \ \ \ \forall \ i,j,k. 
 \end{align}
Let $L$ be a Lagrangian submanifold in $(B^{2n}_R, \omega_0)$ given by a gradient graph 
 \begin{align} \label{eqn dfn lagrangian graph Gamma(u)}
 \Gamma(u):= \{ \Gamma_u (x) : x\in U_R \subset B^n_R \subset \mathbb R^n\},
 \end{align}
 where 
\begin{equation} \label{eqn dfn of Gamma_u}
\Gamma_u(x) = (x, Du(x)).
\end{equation}
for some smooth function $u : U_R \to \mathbb R$, where $U_R \subset B^n_R$. In this case, $L= \Gamma(u)$ satisfies (\ref{dfn HSL for immersion}) if it is a critical point of the area functional 
\begin{align}
\operatorname{Vol} (u) = \int_{U_R} \sqrt{\det g_{ij}} dx,
\end{align}
where $g = \Gamma_u^*h$, among all Hamiltonian variations $t\mapsto u+ t\eta$, where $\eta \in C^\infty_c( U_R)$. The Euler-Lagrange equation has been computed in \cite[Lemma 3.1]{BCW}, which is
 \begin{align} \label{EL for HSL}
 \int_{U_R} \big[ F^{jl} \eta_{jl} + c^k \eta_k \big] dx = 0, \ \ \forall \eta\in C^\infty_c(U_R) 
 \end{align}
 where
 \begin{align}
\label{F^jl formula}
F^{jl} (x, Du, D^2 u) &= \sqrt g g^{ij}\big( ( \delta ^{kl} + \mathcal B_{lk}) u_{ik} + \mathcal C_{li}\big), \\
\label{c^k formula} 
c^k(x, Du, D^2u) &= \frac{1}{2} \sqrt g g^{ij} \big( D_{y_k} \mathcal A_{ij} + 2u_{im} D_{y_k} \mathcal C_{mj} +u_{im} u_{jl} D_{y_k} \mathcal B_{kl}\big),
 \end{align}
 here $\mathcal A$, $\mathcal B$ and $\mathcal C$ are the components of $(h - I_{2n} ) \circ \Gamma_u$ (see (3.2) in \cite{BCW}). Note that $\sqrt g$, $g^{ij}$ can also be represented by $D^2u$, $\mathcal A$, $\mathcal B$ and $\mathcal C$. By (\ref{eqn estimates on g^t}) we obtain the following lemma. 
 
 \begin{lem} \label{F, a bounds by K_j's}
For each $m\in \mathbb N$, there is a constant $C = C(r, m, K_0, \cdots, K_m, \|D^2 u\|_{L^\infty(B^n_r)})$ such that 
 \begin{align}
 \|F^{jl}\|_{C^m(B^n_r)} &\le C,\\
 \|c^k\|_{C^{m-1}(B^n_r)} & \le C, 
 \end{align}
 where $F^{jl}$ and $c^k$ are defined in (\ref{F^jl formula}), (\ref{c^k formula}). 
 \end{lem}

 By (\ref{K_0, K_1 le epsilon_1}) we have $|Dh|<\epsilon_1$. By \cite[Lemma 3.2]{BCW}, if 
 \begin{equation} \label{eqn dfn of epsilon_1}
 \| D^2 u\|_{C^0(B^n_r)}\le \epsilon_1,
 \end{equation}
 then (\ref{EL for HSL}) is $\Lambda$-uniform with $\Lambda = 1/2$. Hence we can apply Proposition \ref{prop Ck alpha estimates} and Lemma \ref{F, a bounds by K_j's} to obtain
 
\begin{thm} \label{Theorem higher order estimates for u satisfying HSL}
Let $u:B^n_r \to \mathbb R$ be a smooth function which satisfies $u(0) = Du(0)= D^2u (0) = 0$ and (\ref{eqn dfn of epsilon_1}). If the graph $\Gamma (u)$ is Hamiltonian stationary Lagrangian in $(B_r^{2n}, \omega, h)$, where $h$ satisfies (\ref{eqn estimates on g^t}) and $K_0, K_1$ satisfy (\ref{K_0, K_1 le epsilon_1}). Then for each $k\ge 3$ and $x\in B^n_{r/2}$,
\begin{align} \label{higher order estimates for u satisfying HSL}
|D^k u(x) |\le C(k, r, K_0, \cdots, K_{k-3}). 
\end{align}
\end{thm}

 \subsection{Curvature estimates, $\epsilon$-regularity}
Next we use the result in the previous subsection to derive local curvature estimates for HSL immersions in $(M, \omega, J, h)$ with small $L^n$-norm of the second fundamental form. 

Using the Darboux coordinates in subsection \ref{subsection darboux with estimates}, one can think of $L\cap \mathscr B_r(p)$ as a Lagrangian immersion in $(B^{2n}_r, \omega_0)$. We first derive an inequality comparing the second fundamental form of $L$ computed using two different Riemannian metrics. 

\begin{prop} \label{prop A^i inequality for different h_i}
 Let $k, K\in \mathbb N$ and $k <K$. Let $L$ be a $k$-dimensional immersed submanifold in an open subset $U$ of $\mathbb R^K$ and let $h_1$ be a Riemannian metrics on $U$ so that $\|h_1- I\|_{C^0(U)} + \|Dh_1\|_{C^0(U)}< \epsilon$ for some positive number $\epsilon<1/2$. Let $A_1, A_0$ be the second fundamental form of $L$ calculated with respect to $h_1$ and the Euclidean metric $h_0 = I$ respectively. Then there is $C = C(K)$ so that 
 $$|A_0(X, X)| \le (1+C\epsilon)|A_1(X, X)|_{h_1} + C\epsilon|X|_{h_1}^2$$ 
 for all $x\in L$ and $X\in T_xL$. 
\end{prop}

\begin{proof}
Let $\nabla^1, \nabla^0$ be the Levi-Civita connections of $h_1, h_0$ respectively. Let $\mathscr C$ be the tensor $\mathscr C(X, Y) = \nabla^1_XY - \nabla^0_XY$. Let $x\in L$ and let $\pi^\perp_\alpha$ be the orthogonal projection onto the orthogonal complement of $T_xL$ with respect to $h_\alpha$ for $\alpha = 1,0$. Let $\{ e_1, \cdots, e_{K}\}$ be a basis of $\mathbb R^{K}$ so that $\{e_1, \cdots, e_k\}$ forms a basis for $T_xL$. Then
$$ \pi^{\perp}_\alpha (Z) = Z - \sum_{a,b=1}^k H^{ab}_\alpha h_\alpha (Z, e_a) e_b, \ \ \ \alpha = 1,0 .$$
Here $(H_\alpha)_{ab} : = h_\alpha (e_a, e_b)$ and $(H^{ab}_\alpha)$ is the inverse matrix of $(H_\alpha)_{ab}$. Hence 
\begin{align*}
    (\pi^\perp_1 - \pi^\perp_0)Z &= \sum_{a,b=1}^k \big( H^{ab}_1 h_1 (Z, e_a) - H^{ab}_0 h_0(Z, e_a) \big) e_b \\
    &=\sum_{a,b=1}^k \big( (H^{ab}_1-H_0^{ab})h_1(Z, e_a) + H^{ab}_0(h_1(Z, e_a) - h_0(Z, e_a))\big)e_b.
\end{align*}
Write $e_a = (e_a^1, \cdots, e_a^K)$ and $Z = (Z^1, \cdots, Z^K)$, then 
\begin{align*}
    (H_\alpha)_{ab} = \sum_{i,j=1}^K e_a^i e_b^j (h_\alpha)_{ij}, \ \ h_\alpha (Z, e_a) &= \sum_{i,j=1}^K Z^i e_a^j (h_\alpha)_{ij}. 
\end{align*}
This implies 
\begin{align*}
    |H_1^{ab} - H_0^{ab}|&\le C \|h_1 - h_0\|_{C^0(U)}, \\
    |h_1(Z, e_a)-h_0(Z, e_a)| &\le C \| h_1 - h_0\|_{C^0(U)} |Z|
\end{align*}
and
\begin{equation} \label{eqn norm bound of pi_1^perp - pi^perp_2}
| (\pi^\perp_1 - \pi^\perp_2)Z| \le C\|h_1-h_0\|_{C^0(U)} |Z|.
\end{equation}
for some $C = C(K)$. Let $X, Y$ be any tangent vector fields of $L$ defined on some open neighborhood of $x$. Then 
\begin{align*}
    A_0(X, Y) &= \pi^\perp_0 (\nabla^0_XY)  \\
    &= \pi^\perp_0 (-\mathscr C(X, Y) +\nabla^1_XY) \\
    &= -\pi^\perp_0\mathscr C(X, Y) +\pi^\perp_1 \nabla^1_XY + (\pi^\perp_0- \pi^\perp_1) (\nabla^1_XY) \\
    &= A_1(X, Y) - \pi^\perp_0\mathscr C(X, Y)+(\pi^\perp_0- \pi^\perp_1) (\nabla^1_XY). 
\end{align*}
Now let $X$ be chosen such that $\nabla^1 _XX = A_1(X, X)$ at $x$. Using (\ref{eqn norm bound of pi_1^perp - pi^perp_2}), at $x$ we have
\begin{align*}
     |A_0(X, X)| &\le |A_1(X, X)| + |\mathscr C(X, X)| + C\|h_1-h_2\|_{C^0(U)} |A_1(X, X)|\\
     &\le (1+C\epsilon) |A_1(X, X)|_{h_1} + |\mathscr C (X, X)|. 
\end{align*}
For some $C= C(K)$. To estimates $\mathscr C (X, X)$, note that locally
$$\mathscr C_{ij}^k = (\Gamma_1)_{ij}^k - (\Gamma_0)_{ij}^k = (\Gamma_1)_{ij}^k, $$
where $(\Gamma_1)_{ij}^k$ are the Christoffel symbols of $h_1$. Since $\|Dh_1\|_{C^0(U)}<\epsilon$ by assumption and $(\Gamma_1)_{ij}^k = h^{-1}_1 * Dh_1$, 
$$ |\mathscr C(X, X)|\le C\epsilon |X|_{h_1}^2$$
for some $C = C(K)$. This finishes the proof of Proposition \ref{prop A^i inequality for different h_i}. 
\end{proof}

\begin{lem} \label{lem graphical C^2 estimets}
Let $U$ be an open subset in $\mathbb R^{2n}$ and let $L$ be an immersed Lagrangian submanifold in $(U, \omega_0, h, J)$, so that $h$ satisfies (\ref{eqn estimates on g^t}) and the $C^0$-norm of the second fundamental form of $L$ is bounded by $C_A$. Then there is $C_1$ depending on $n, K_0, K_1$ so that the following holds: for any $\epsilon \in (0,1]$, write $r = \epsilon/(C_1(C_A+1))$. For any $p\in L\cap U$ so that $B^{2n}_r (p) \subset U$, any embedded connected component $L^i$ of $L\cap B^{2n}_ r (p)$ containing $p$ can be written as a gradient graph $\Gamma_{u}$, where $u$ is a function defined on $U_i\subset T_p L^i$,  $B^n_{r/2}\subset U_i $ and $u$ satisfies 
\begin{equation} \label{eqn u (0) = Du(0) = D^2 u(0) = 0}
u (0) = Du(0) = D^2 u (0) = 0
\end{equation}
and 
\begin{equation} \label{eqn |D^2 u (x)| le epsilon}
    \|D^2 u \|_{C^0(U_i)} \le \epsilon.
\end{equation}
\end{lem}

\begin{proof}
Apply Proposition \ref{prop A^i inequality for different h_i} with $h_1 = h$, we conclude that $|A_0|\le c_2 (C_A +1)$, where $A_0$ is the second fundamental form of $L$ in $(U, h_0)$ and $c_2$ depends only on $n, K_0, K_1$. By \cite[Theorem 2.6]{B} (see also \cite[Theorem 2.4]{Langer}), there is a dimensional constant $c$ so that if $L^i$ is an embedded connected component of $L\cap B^{2n}_{c\epsilon (C(C_A+1))^{-1}}(p)$ containing $p$, then $L^i$ is graphical: there is $X_i : U_i \to \mathbb R^n$, where $B^{n}_{{c\epsilon (2c_2(C_A+1))^{-1}}} (p) \subset U_i$, so that 
$$ L^i = \{ (x, X_i(x) ) : x\in U_i\}$$
with $X_i(0) = DX_i (0) = 0$ and $|X_i(x)|\le \epsilon$ for all $x\in U_i$. Since $L^i$ is Lagrangian, $X_i = Du_i$ for some smooth function $u_i : U_i \to \mathbb R$ with $u_i(0)=0$. This finishes the proof of the lemma, by choosing $C_1 =2c_2/c$. 
\end{proof}

The following lemma can be proved by a direct computation, using the bounds on $D^2 u$, $h-I$ and $Dh$.
 \begin{lem} \label{lem graphical C^3 estimets}
With the same assumption as in Lemma \ref{lem graphical C^2 estimets}, there is $C_2>0$ depending only on $n$, $K_0, K_1$ such that
 \begin{align}
|D^3 u(x)| \le C_2(| A(x)| +1), \ \ \ |A(x)| \le C_2( |D^3 u(x)| +1) 
 \end{align}
 whenever $|x|\le  \epsilon_1 (C_1(C_A +1))^{-1}$.  
 \end{lem}

From the above lemmas we obtain the following proposition. 
 
\begin{prop} \label{prop split into embedded components}
Let $L$ be a properly immersed Lagrangian submanifold in $M$. Let $p\in M$ and let $\delta, C_A$ be positive numbers. Assume that the $C^0$-norm of the second fundamental form $A$ of $L$ is bounded above by $C_A$ in the Darboux ball $\mathscr B_\delta (p)$. Then there is $C_2>0$ depending only on $n, K_0, K_1$ such that the following holds: for any $\epsilon \in (0,\epsilon_1]$, $L\cap \mathscr B_r(p)$ splits into embedded connected components, where $r= \min\{ \delta/2, \epsilon (C_2(C_A+1))^{-1}\}$. Moreover, for each $v\in \mathbb U_p$, each embedded connected component $L_i$ of $L\cap \mathscr B_r(p)$, $z\in \widetilde L^i := \Upsilon_{p,v}^{-1}(L_i)$, up to a unitary action, $\widetilde L^i$ is the gradient graph of a function $u_i$ with $z=(0, 0)$,
\begin{itemize}
    \item [(i)] $u_i(0) = Du_i(0) = D^2u_i(0)$,
    \item [(ii)] $|D^2 u_i(x)|\le \epsilon$, and 
    \item [(iii)] $|D^3 u_i(x)|\le C_1 (C_A+1)$. 
\end{itemize} 
\end{prop}

\begin{proof}
For any $r < \delta$ and $v\in \mathbb U_p$, let $\tilde L := \Upsilon_{p, v}^{-1} (L\cap \mathscr B_p(r))$. Then $\tilde L$ is a immersed Lagrangian submanifold in $B^{2n}_r$. For any $\epsilon \in (0, \epsilon_1]$, let $r_1(\epsilon) = \epsilon/ (C_1(C_A+1))$, where $C_1$ is defined in Lemma \ref{lem graphical C^2 estimets} and let $r_2 = r_1(\epsilon)/4$. Then for all $q\in B^{2n}_{r_2}$ we have $B^{2n}_{r_2}\subset  B^{2n}_{2r_1}(q) \subset B^{2n}_{r_1}$. By Lemma \ref{lem graphical C^2 estimets}, when $r = r_1(\epsilon)$, $\tilde L$ splits into finitely many embedded connected components. Let $\bar L^1, \cdots , \bar L^I $ be those components that intersects $B^{2n}_{r_2}$. For any $i = 1, \cdots, I$ and $z\in \bar L^i\cap B^{2n}_{r_2}$, apply Lemma \ref{lem graphical C^2 estimets} to the ball $B^{2n}_{2r_2}(z)$. Thus up to a unitary action, $\bar L^i \cap B^{2n}_{2r_2}$ is the gradient graph of a function $u_i$. (i), (ii), (iii) follows from Lemma \ref{lem graphical C^2 estimets} and Lemma \ref{lem graphical C^3 estimets}. Since $B^{2n}_{r_2} \subset B^{2n}_{2r_2}$, the Proposition is proved by choosing $C_2 = 4C_1$. 
\end{proof}

 The following $\epsilon$-regularity theorem is essential to the proof of Theorem \ref{main thm}. 

 \begin{thm} \label{thm epsilon regularity}
 There are positive numbers $\epsilon_0, C_{\epsilon_0}$ depending only on $K_0, K_1, K_2$ such that the following holds: if $L$ is an $n$-dimensional properly immersed HSL submanifold in a symplectic manifold $M$, $p\in L$ and
 \begin{align} \label{eqn L^n norm of A <epsilon_0}
 \int_{L \cap \mathscr B_{r_0}(p) } |A| ^n d\mu _L\le\epsilon_0, 
 \end{align}
where $r_0\le 1$. Then for all $0<\sigma \le r_0 $ and $y\in \mathscr B_{r_0 - \sigma}(p)\cap L$, 
 \begin{align}
 \sigma |A(y)| \le C_{\epsilon_0}. 
 \end{align}
 \end{thm}
 
 \begin{proof}
Let $L$ be given by an immersion $\iota: N\to M$. Let $\Upsilon_{p,v}: B^{2n}_{r_0}\to M$ be given by Proposition \ref{D coord. with estimates}, where $v\in \mathbb U_p$. Let $z_0$ be the maximum of the function 
$$z\mapsto (r_0-|z|)^2 \max_{s \in \iota^{-1}(\Upsilon_{p,v} (z))}  |A(s)|^2$$ 
defined on $\Upsilon_{p,v}^{-1} (L\cap \mathscr B_{r_0}(p))$. Note that maximum exists since $\iota^{-1}(\Upsilon_{p,v} (z))$ is finite for each $z$. We assume this maximum is positive, or otherwise the result is trivial. So $|z_0|<r_0$. Let $s_0 \in \iota^{-1} (\Upsilon_{p,v} (z_0))$ such that 
$$|A(s_0)|^2 = \max_{s \in \iota^{-1}(\Upsilon_{p,v} (z_0))} |A(s)|^2.$$
For all $z\in B^{2n}_{\frac{r_0-|z_0|}{2}}(z_0)$ and $s\in \iota^{-1} (\Upsilon_{p,v} (z))$,  
\begin{align*}
|A(s)|^2 &\le \frac{(r_0-|z_0|)^2}{(r_0-|z|)^2} |A (s_0)|^2 \\
&\le \frac{(r_0-|z_0|)^2}{\left(\frac{r_0-|z_0|}{2}\right)^2} |A (s_0)|^2 \\
&\le 4 | A (s_0)|^2.
\end{align*}
Let $p_0 = \Upsilon_{p,v}(z_0)$. By (\ref{comparing geodesic ball and darboux ball}), $U:=\mathscr B_{\frac{r_0-|z_0|}{8}} (p_0) \subset \Upsilon_{p,v} (B^{2n}_{\frac{r_0 - |z_0|}{2}} (z_0))$. By Proposition \ref{prop split into embedded components}, there is $C_1>0$ such that $L \cap U\cap \mathscr B_{r} (p_0)$ splits into embedded connected components, where $r  =(C_1(2|A(s_0)|+1))^{-1}\epsilon_1$ and each of the components is HSL by Corollary \ref{cor HSL on each embedded connected component}. Moreover, let $L_i$ be a connected component so that $\iota^{-1}(L_i)$ contains $s_0$. Choose $v_0 \in \mathbb U_{p_0}$ so that $v_0 (\mathbb R^n \times \{0\}) = T_{p_0} L_i$. Then $\Upsilon_{p_0, v_0}^{-1} L_i$ is the gradient graph of a smooth function $u$ with $u(0) = Du(0) = D^2 (0) = 0$ and 
$$|D^2u(x)|\le \epsilon_1, \ \ \ |D^3u(x)|\le C_1 (2|A(s_0)|+1)$$  
for all $x$ in the domain of $u$. 

Let 
$$R= \frac{r_0-|z_0|}{8} |A(s_0)|, \  \ \ t=\frac{1}{ |A(s_0)|}$$
 and let 
$$t: B^{2n}_R \to B^{2n}_{\frac{r_0-|z_0|}{8}}$$
be the scaling $z\mapsto tz$. Then $\tilde L= (\Upsilon_{p_0, v_0} \circ t)^{-1}(L_i)$ is an immersed HSL in $(B^{2n}_R, \omega_0,\tilde h)$ given by 
$$\tilde \iota = t^{-1} \circ \iota: V \to B^{2n}_R$$ 
with metric $\tilde h = t^{-2} ( \Upsilon_{p_0, v_0} \circ t)^* h$. By the choice of $R$, the second fundamental form $\widetilde A$ of $\widetilde L$ satisfies $\|\widetilde A\|_0\le 2$ and $|\widetilde A(0)|=1$. Also, $\widetilde L$ is a gradient graph of a smooth function $\tilde u$ with 
$$|D^2\tilde u(x)|\le \epsilon_1, \ \ \ |D^3 \tilde u(x)|\le \frac{C_1(2|A(s_0)|+1)}{|A(s_0)|} \le 3C_1 $$ 
for all $x\in B^{2n}_{(3C_1)^{-1}}\cap B^{2n}_R$. Moreover,  
\begin{align}
\int_{\widetilde L} |\widetilde A|^n  d\mu_{\tilde L} \le \int_{B_{\frac{r_0-|z_0|}{8}}} |A|^n  d\mu <\epsilon_0
\end{align}
since the quantity is scale-invariant. Using Proposition \ref{prop Ck alpha estimates} with $m=4$ and (\ref{F, a bounds by K_j's})
$$|D^4\tilde u(x) |\le C(K_0, K_1, K_2)$$
for all $x\in B^{2n}_R$ with $|x|\le (4C_1)^{-1}$. This implies 
$$|\widetilde \nabla \widetilde A| \le C(K_0, K_1, K_2)$$
for all $x\in B^{2n}_R$ with $|x|\le (4C_1)^{-1}$. Since $|\widetilde A (0)| = 1$, one concludes 
$$ |\widetilde A (x)|\ge \frac 12$$
for all $x\in \widetilde L^1$ with $|x| \le \tilde r$, where $\tilde r = \tilde r (K_0, K_1, K_2)$. Now choose $\epsilon _0 = \omega^{-1}_n (4\tilde r)^n$. If $R> \tilde r$, then 
$$ \int_{B^{2n}_R} |\tilde A|^n d\mu_{\tilde L} \ge \int_{B^{2n}_{\tilde r}} |\tilde A|^n d\mu_{\tilde L} \ge \frac{1}{2^n} \mu_{\tilde L} (B^{2n}_{\tilde r})\ge \frac{1}{\omega_n (4\tilde r)^n} =\epsilon_0,$$
which is a contradiction. Thus 
$$R= \frac{r_0-|z_0|}{8}|A(s_0)| \le \tilde r,$$
which implies for all $z\in \mathscr B_{r_0}(p)$ and $s\in \iota^{-1}(\Upsilon_{p,v} (z))$, 
$$ |A(s)|\le \frac{8\tilde r}{r_0-|z|}.$$
 \end{proof}

 \section{Proof of Theorem \ref{main thm} for $n=1$} \label{pf of main thm n=1}
 In this section we prove Theorem \ref{main thm} for $n=1$. The general case is proved in the next section. We remark that it is necessary to split the proof into two cases: when $n\ge 2$ we use \cite[Corollary 4.5]{CW2}, which says that if $L$ is a Hamiltonian stationary Lagrangian in $M\setminus S$ and $S$ is finite, then $\overline L$ is also Hamiltonian stationary Lagrangian varifold in $M$. As pointed out in the introduction in \cite{CW2}, this does not hold when $n=1$.  
 
 \begin{proof}[Proof of Theorem \ref{main thm} for $n=1$] In this case, $(L_k)$ is a sequence of compact connected immersed curves in an oriented Riemannian surface $(M, h, J)$ and the Lagrangian condition is automatically satisfied by any smooth curve.  For each $k$, $L_k$ is given by an immersion $\gamma_k : \mathbb S^1 \to M$. We assume each $\gamma_k$ to be parameterized proportional to arc length.
 
 Since $J$ is compatible to $h$, one checks that an immersed curve $\gamma : \mathbb S^1 \to M$ in $(M, h)$ is Hamiltonian stationary if and only if its curvature is constant. We remark that if $h$ is generic then $(M, h)$ admits infinitely many embedded circles with constant curvatures \cite{Ye}. 
 
 For each $k\in \mathbb N$, let $\kappa_k$ be the (constant) curvature of $\gamma_k$. Let $\operatorname{Length}(L)$ denotes the length of the immersed curves $L$ defined in (\ref{eqn dfn of volume of L}). The condition (\ref{eqn bounded area and norm of A}) implies  
\begin{equation} \label{Bound assumption on length for n=1} 
\operatorname{Length} (\gamma_k)=\operatorname{Length}(L_k)\le C_V 
\end{equation}
and 
\begin{equation} \label{Bound assumption on total curvature for n=1} 
|\kappa_k| \operatorname{Length}(\gamma_k) =\int_{L_k} |\kappa_k| \le C_A. 
\end{equation}

%First we derive a lower bound for the total curvature for short closed curves. Since $(M, h)$ is compact, there is $\epsilon >0$ small so that any curve $\gamma$ in $M$ of length less than $\epsilon$ lies in an isothermal coordinates neighborhood $U$ of $M$, where $h|_U = e^{2u} \delta_{ij}$ is conformal and $\|u\|_{C^1 (U)} \le C$, where $C$ is independent of $\epsilon$. Let $\kappa^0$ be the curvature of $\gamma$ calculated with respect to the Euclidean metric $\delta_{ij}$ within this chart. By a direct calculation, 
%$$ \int_\gamma |\kappa| \ge \int_\gamma |\kappa^0| - C\operatorname{Length}_h(\gamma).$$
%If $\gamma$ is a closed curve in this chart, $\int_\gamma |\kappa^0| \ge \int_\gamma \kappa^0 \ge 2\pi$ and thus 
%\begin{equation}\label{total curvature of short closed curve ge pi}  \int_\gamma |\kappa| \ge \pi.
%\end{equation}
%\begin{lem} \label{ell is uniformly bounded}
%$m_k \le \lfloor {C_V/\epsilon}\rfloor+ C_A/\pi$ for all $k\in \mathbb N$.
%\end{lem}
%\begin{proof}[Proof of Lemma \ref{ell is uniformly bounded}] By (\ref{Bound assumption on length for n=1}), there are at most $\lfloor {C_V/\epsilon}\rfloor$ components of $L_k$ with length larger than $\epsilon$. Thus more than $m_k - \lfloor {C_V/\epsilon}\rfloor$ component has length less than $\epsilon$ and thus $\pi (m_k - \lfloor {C_V/\epsilon}\rfloor) \le C_A$ by (\ref{Bound assumption on total curvature for n=1}) and (\ref{total curvature of short closed curve ge pi}).
%\end{proof}
%By the above lemma and 

Taking a subsequence of $(L_k)$ if necessary, we may assume that 
\begin{equation}\label{limit of length of each component}
\lim_{k\to \infty} \operatorname{Length} (\gamma_k) = \ell _\infty \in [0,C_V]
\end{equation}
and 
\begin{equation}\label{limit of total curvature of each component}
\lim_{k\to \infty} \kappa_k = \kappa _\infty \in [-\infty,\infty].
\end{equation}
By compactness of $M$, we may also assume that
\begin{equation}\label{limit of gamma j k at 0}
\lim_{k\to\infty} \gamma_k(0) = p, \ \ \ \lim_{k\to\infty} \gamma_k'(0) = v 
\end{equation}
for some $p \in M$ and $v \in T_{p} M$. 

If $\ell_\infty = 0$, then $(\gamma_k)$ converges to the point $p$. If $\ell_\infty >0$, then $\kappa_\infty \neq \pm \infty$ by (\ref{Bound assumption on total curvature for n=1}). Using (\ref{limit of gamma j k at 0}) and the smooth dependence of ODE on initial data and parameters \cite[Theorem 4.1]{Hartman}, $(\gamma_k)$ converges smoothly to an immersed curve $\gamma_\infty$ with $\gamma_\infty(0) = p$, $\gamma_\infty'(0) = v$ and constant curvature $\kappa_\infty$. This finishes the proof of Theorem \ref{main thm} when $n=1$.

%Let 
%\begin{align*} 
%S = \{ p_j :  \ell^j_\infty = 0\}, \ \  S' = \{p_j :\ell^j_\infty > 0\}
%\end{align*}
%and let 
%\begin{equation} \label{definition of L_infty away from S when n=1}
%L_\infty = \bigcup _{p_j \in S'} \operatorname{Im} (\gamma^j_\infty) \setminus S.
%\end{equation}
%Then $S$ is finite and $(L_k)$ converges smoothly to $L_\infty$ in $M\setminus S$. Also, $L_\infty$ is a union of immersed curves in $M\setminus S$ with constant curvature and $\overline L_\infty = \cup_{p_j \in S'} \operatorname{Im} (\gamma^j_\infty)$ is smooth and satisfies (\ref{mu_L satisfies HSL}) and 

\end{proof}

\section{Proof of Theorem \ref{main thm} for $n\ge 2$} \label{pf of main thm n>1}
First we prove a simple covering lemma by Darboux balls. 
\begin{prop} \label{prop covering theorem for darboux balls}
 Let $\lambda \in (0,1/4)$. There is $b$ depending only on $n$ and $\lambda$ such that the following holds. For any $r <1$, there is a finite open covering $\mathscr U$ of $M$ by Darboux balls $\mathscr B_r(a)$ so that (i) each member of $\mathscr U$ intersect with at most $b-1$ other members of $\mathscr U$, and (ii) the open cover $\{ \mathscr B_{\lambda r} (p) : \mathscr B_r(p)\in \mathscr U\}$ still covers $M$. 
 \end{prop}

%We start with the following lemma. 
%\begin{lem}Let $\lambda <1/4$. Then there is $b=b(n, \lambda)$ so that for every $r<1$ and $p\in M$, the Darboux ball $\mathscr B_r(p)$ is covered by $b$ Darboux balls of radius $\lambda r$. 
%\end{lem}

%\begin{proof}
%Recall that $\mathscr B_r(p)$ is the image of $\Upsilon_{p, v} :B^{2n}_r\to M$ for some $v\in \mathbb U_p$ and $\Upsilon_{p, v}$ is defined on $B^{2n}_{2}$ so that $h$ (or more precisely, $\Upsilon_{p,v}^*h$) satisfies $0.5 h_0 \le h \le 2 h$ on $\mathscr B^{2n}_{2}$, where $h_0$ is the Euclidean metric. Now cover $B^{2n}_r$ by $b_0=b_0(n,\lambda)$ many Euclidean balls of radius $\lambda r/4$ in $B^{2n}_{2}$. Note that $b_0$ depends on $n, \lambda$ but not on $r$. By (\ref{comparing geodesic ball and darboux ball}), we have shown that $\mathscr B_r(p)$ is covered by $b$ many Darboux balls of radius $\lambda r$. 
%\end{proof}

\begin{proof} Let $\mathscr V= \{ \mathscr B_{\lambda r/8} (p_1), \cdots, \mathscr B_{\lambda r/8} (p_N)\}$ be a maximal collection of disjoint Darboux balls of radius $\lambda r /8$ in $M$. Then $\{ \mathscr B_{\lambda r} (p_1), \cdots, \mathscr B_{\lambda r} (p_N)\}$ covers $M$: to see this, assume $p\in M$ and is not in one of $\mathscr B_{\lambda r/8} (p_i)$. Since $\mathscr V$ is maximal, there is $i$ so that $\mathscr B_{\lambda r/8} (p) \cap \mathscr B_{\lambda r/8} (p_i)$ contains an element $y$. By (\ref{comparing geodesic ball and darboux ball}) and triangle inequality, 
$$ d(p, p_i) \le d(p, y) + d(y, p_i)< \frac{\lambda r}{4} + \frac{\lambda r}{4} = \frac{\lambda r}{2}.$$
Hence $p\in B^M_{\lambda r/2} (p_i) \subset \mathscr B_{\lambda r} (p_i)$ by (\ref{comparing geodesic ball and darboux ball}) again. Let $\mathscr U$ be the open cover 
$$ \mathscr U = \{ \mathscr B_r(p_1), \cdots, \mathscr B_r(p_N)\}.$$ Note that $\mathscr U$ satisfies (ii) by construction. To show (i), let $p_i$ be fixed and assume that $y_{i_j}$ lies in the intersection of $\mathscr B_r (p_i)$ and $\mathscr B_r(p_{i_j})$, for some $j=1, \cdots, k$. Then $d(p_i, p_{i_j})\le 4r$, or $p_{i_j} \in \mathscr B_{8r} (p_i)$ for all $i_j$. However, since $\mathscr V$ is a collection of disjoint Darboux balls with radius $\lambda r/8$, one has $d(p_{i_j}, p_{i_l}) \ge \lambda r/8$. Let $v_i \in \mathbb U_{p_i}$ be fixed and let $z_j = \Upsilon_{p_i, v_i}^{-1} (p_{i_j})$. Then $\{z_1, \cdots, z_k\}$ be a collection of points in $B^{2n}_{8r}$ so that $|z_i-z_j| \ge \lambda r/16$ whenever $i\neq j$. This implies $k \le b$, where $b$ depends only on $n$, $\lambda$ but not on $r$. This finishes the proof of (i). 
\end{proof}

\begin{proof}[Proof of Theorem \ref{main thm} for $n\ge 2$] Choose $\lambda = \frac{\epsilon_1}{27C_2C_{\epsilon_0}}$, where $\epsilon_1$, $C_2$ and $C_{\epsilon_0}$ are defined in (\ref{eqn dfn of epsilon_1}), Proposition \ref{prop split into embedded components} and Theorem \ref{thm epsilon regularity} respectively. By Proposition \ref{prop covering theorem for darboux balls}, there is a constant $b = b(n, \lambda)$ so that the following holds: for any $m\in \mathbb N$, there is a finite open cover 
$$\mathscr U_m = \{\mathscr B_{2^{-m}} (\tilde p^m_1), \cdots,  \mathscr B_{2^{-m}} (\tilde p^m_{j(m)})\}$$
of $M$ by Darboux balls of radius $2^{-m}$, so that
\begin{enumerate}
\item [(i)] each elements in $\mathscr U$ intersects at most $b-1$ other members in $\mathscr U$, and
\item [(ii)] the collection 
$$\mathscr U_{m, \lambda} = \{\mathscr B_{\lambda 2^{-m}} (\tilde p^m_1), \cdots,  \mathscr B_{\lambda 2^{-m}} (\tilde p^m_{j(m)})\}$$
still covers $M$.
\end{enumerate} 
The first condition implies that for each $k \in \mathbb N$,
\begin{align} \label{sum of L n norm of |A| over balls} \sum_{j=1}^{j(m)} \int_{\mathscr B_{2^{-m}} (\tilde p^m_j) \cap L_k }|A_k |^n d\mu_k < b \int_{L_k} |A_k|^n d\mu_k < bC_A,
\end{align}
where $C_A$ is give in Theorem \ref{main thm} and $\mu_k := \mu_{L_k}$. For any fixed $k$ and $m$,  let $J_{k,m} =\{ p^m_{1,k}, \cdots, p^m_{\ell, k}\}$ be the subset of $\{\tilde p^m_1, \cdots, \tilde p^m_{j(m)}\}$ so that 
$$ \int_{\mathscr B_{2^{-m}} (p^m_{j, k}) \cap L_k} |A_k |^n d\mu_k \ge \epsilon_0.$$
Here $\epsilon_0$ is given in Theorem \ref{thm epsilon regularity} and $\ell = \ell(m, k)$ is less than $bC_A/\epsilon_0$ by (\ref{sum of L n norm of |A| over balls}). Taking a diagonal subsequence of $(L_k)$ if necessary, we may assume that $\ell(k, m) = \ell(m)$ and $J_{k,m} = J_m$ are both independent of $k$ for all $m\in \mathbb N$. Write 
$$ J_m = \{ p^m_1, \cdots, p^m_{\ell(m)}\}.$$
Using $\ell (m) <bC_A/\epsilon_0$, there is a subsequence $(m_i)$ such that $\ell (m_i) = \ell$ for all $i\in \mathbb N$. Using the compactness of $M$ and taking a further subsequence of $(m_i)$ if necessary, we may assume that 
$$ \lim_{i \to \infty} p^{m_i}_j = p_j$$
for all $j =1, \cdots, \ell$. 

Let $S = \{ p_1, \cdots, p_\ell\}$ and $r_i = 2^{-m_i}$. Fix any $p \in \{ \tilde p^{m_i}_1, \cdots , \tilde p^{m_i}_{j(m)}\} \setminus J_{m_i}$. By definition of $J_{m_i}$ this implies  
%$\mathscr B_i = \cup_{j} \mathscr B_{r_i} (p^{m_i}_j)$ and $U_i = M \setminus  \overline {\mathscr B_i}$. 

%Picking a further subsequence of $(m_i)$ if necessary, we may find a sequence of open subsets $(V_i)$ of $U_i\setminus S_{1/i}$ so that 
%$$ \overline V_i \subset V_{i+1}, \ \ \ i\in \mathbb N$$
%and 
%$$ \bigcup_i V_i = M \setminus  S. $$
$$ \int_{\mathscr B_{r_i} (p) \cap L_k} |A_k|^n d\mu_k < \epsilon_0$$
for all $k\in \mathbb N$. By Theorem \ref{thm epsilon regularity},
\begin{equation} \label{eqn |A| le C/r in B_r}
|A_k (q)| \le 2C_{\epsilon_0}/r_i,  \ \ \ \forall k \in \mathbb N \text{ and } q\in\mathscr B_{r_i/2} (p). 
\end{equation}
Setting $C_A = 2C_{\epsilon_0}/r_i$ and note that 
$$ \frac{\epsilon_1}{C_2 (C_A +1)} = \frac{r_i\epsilon_1 }{C_2 (2C_{\epsilon_0}+r_i)} \ge \frac{\epsilon_1}{3C_2C_{\epsilon_0}}r_i = 9\lambda r_i.$$
By Proposition \ref{prop split into embedded components}, for each $k\in \mathbb N$, $\mathscr B_{8\lambda r_i} (p)\cap L_k$ splits into finitely many embedded connected components 
$$\mathscr B_{8\lambda r_i} (p)\cap L_k = L_k^1\cup \cdots \cup L_k^{I_k}.$$
Let $v\in \mathbb U_p$ be fixed and for each $a=1, \cdots, I_k$, denote $\tilde L_k^a := \Upsilon_{p,v}^{-1} (L^a_k)$. 
\begin{lem} \label{lem M_i existence}
For each $i\in \mathbb N$, there is $M_i\in \mathbb N$ depending on $i$, $\lambda$, $n$ and $C_V$ in (\ref{eqn bounded area and norm of A}) only so that $\mathscr B_{8\lambda r_i} (p)\cap L_k$ has at most $M_i$ embedded connected components which intersect $\mathscr B_{\lambda r_i}(p)$. 
\end{lem}

\begin{proof}[Proof of Lemma \ref{lem M_i existence}]
Let $L^a_k$ be one of the components $\mathscr B_{8\lambda r_i} (p)\cap L_k$ which intersects $\mathscr B_{\lambda r_i}(p)$. Let $z^a_k \in B^{2n}_{\lambda r_i}\cap \tilde L_k^a$. Then there is a unitary matrix $A^a_k$ and a smooth function $u^a_k$ defined on $U^a_k$ so that 
$$ (\Upsilon^a_k)^{-1} (\tilde L_k^a \cap B^{2n}_{8\lambda r_i} (z^a_k))= \{ (x, Du^a_k(x)) : x\in U^a_k\}, $$
where $\Upsilon^a_k (z) = A^a_k z + z^a_k$. Using (\ref{eqn estimates on g^t}), we obtain $\mu(L^a_k) \ge C(n) (\lambda r_i)^n$ for some dimensional constant $C(n)$. Together with the assumption on the area bound in (\ref{eqn bounded area and norm of A}), the lemma is proved. 
\end{proof}

%Next we pick finitely many points $q_1, \cdots, q_{n_i} \in V_i$ so that 
%$$V_i \subset B^M_{C_3 r_i/8} (q_1) \cup \cdots \cup B^M_{C_3 r_i/8} (q_{n_i}).$$ 
By Lemma \ref{lem M_i existence} and taking a diagonal subseqeunce of $(L_k)$ if necessary, we may assume that for any $i\in \mathbb N$ and $p \in \{ \tilde p^{m_i}_1, \cdots, \tilde p^{m_i}_{j(m)}\} \setminus J_{m_i}$, there is an integer $n_i(p)$ such that $\mathscr B_{8\lambda r_i} (p) \cap L_k$ has exactly $n_{i}(p)$ connected components which intersect $\mathscr B_{\lambda r_i} (p)$, where $n_{i}(p) \le M_i$. Let $L^1_k, \cdots, L^{n_i(p)}_k$ be such connected components. 

Let $z^a_k$, $A^a_k$, $\Upsilon^a_k$ and $u^a_k : U^a_k\to \mathbb R$ be defined as in the proof of Lemma \ref{lem M_i existence}, where $u=u^a_k$ satisfies (\ref{eqn u (0) = Du(0) = D^2 u(0) = 0}) and (\ref{eqn |D^2 u (x)| le epsilon}) with $\epsilon = \epsilon_1$. By Theorem \ref{Theorem higher order estimates for u satisfying HSL}, for any $m\ge 3$, there are $\overline C_m$ depending on $r_i,\lambda, m, K_0, \cdots, K_{m-3}$ so that 
\begin{align} \label{higher order derivative bound of u}
    |D^m u^a_k (x)|\le \overline C_m, \ \ \text{ for all } x\in B^n_{4 \lambda r_i}.
\end{align}

Taking a diagonal subsequence of $(L^k)_{k=1}^\infty$ if necessary, we may assume that for each $i\in \mathbb N$ and $p$, 
\begin{enumerate}
\item the sequence of points $(z^a_k)_{k=1}^\infty$ converges to $ z^a \in \overline {B^{2n}_{\lambda r_i}}$ as $k\to \infty$, 
\item the sequence of unitary matrices $(A^a_k)_{k=1}^\infty$ converges to some $A^a \in U(n)$ as $k\to \infty$, and
\item the sequence of functions $(u^a_k)_{k=1}^\infty$ converges in $C^m(B^n_{3\lambda r_i})$ for all $m\in \mathbb N$ to a smooth function $u^a : B^n _r \to \mathbb R$. This is possible by the higher order estimates (\ref{higher order derivative bound of u}). 
\end{enumerate}
Then the sequence of immersions $(\Upsilon^a_k \circ \Gamma_{u^a_k})_k$ converges to the immersion $\Upsilon^a \circ \Gamma_{u^a}$, where $\Upsilon^a(z) = A^az+z^a$. In particular, there is $\tilde u^a_k$ defined on $B^n_{2\lambda r_i}$ so that $\tilde L^a_k$ is locally given by the immersion $\Upsilon^a \circ \Gamma_{\tilde u^a_k}$.

Next we define the Lagrangian integral $n$-varifold $L$ on $M\setminus S$. Let 
$$\mathscr B = \{\mathscr B_{\lambda r_i} (p) : i\in \mathbb N, \ p \in \{\tilde p^{m_i}_1, \cdots, \tilde p^{m_i}_{j(m_i)}\}\setminus J_{m_i}\}.$$
Clearly $\mathscr B$ is an open cover of $M\setminus S$. For each $K\in \mathbb N$, let 
$$\mathscr B_K = \{ \mathscr B_{\lambda r_i} (p) \in \mathscr B: \mathscr B_{\lambda r_i} (p) \cap S_{1/K} = \emptyset\},$$
where for any $\rho >0$ we set $S_{\rho} = \{ q\in M : d^M(q, s)<\rho  \text{ for some } s\in S\}$. Then $\tilde{\mathscr B} = \cup_K \mathscr B_K$ is an open cover of $M\setminus S$. For each $q\in M\setminus S$, then $q\in \mathscr B_{\lambda r_i} (p)$ for some  $U=\mathscr B_{\lambda r_i} (p)\in \tilde{\mathscr B}$. Define $|L|$ on $\mathscr B_{\lambda r_i} (p)$ by 
$$|L|= |L^1|+ \cdots + |L^{n_i(p)}|,$$ 
where for each $a = 1, \cdots, n_i(p)$, $L^a$ is the limits of $(L^a_k\cap \mathscr B_{\lambda r_i} (p))_{k=1}^\infty$ constructed in the previous paragraph. Since $z^a \in B^{2n}_{\lambda r_i}$ for each $a =1, \cdots, n_i(p)$, $B^{2n}_{\lambda r_i}\subset B^{2n}_{2\lambda r_i} (z^a)$. Hence $(L^a_k\cap \mathscr B_{\lambda r_i} (p))_{k=1}^\infty$ converges graphically in $\mathscr B_{\lambda r_i}(p)$ to $L\cap U$ in $C^\infty$-topology. Since $q\in M\setminus S$ is arbitrary, $(L_k)_{k=1}^\infty$ converges in $M\setminus S$ smoothly locally graphically to $L$. Since the convergence is locally smooth and each $L_k$ is HSL, $L$ is also HSL on $M\setminus S$. By \cite[Corollary 4.4]{CW2}, $\overline L$ admits a structure of an $n$-integral Lagrangian varifold so that (\ref{EL for HSL}) holds (although Theorem 4.3 and Corollary 4.4 in \cite{CW2} are stated only for Lagrangian immersions in $\mathbb R^{2n}$, by isometrically embed $(M, h)$ into some Euclidean space, the proofs work for immersions in $(M, \omega, J, h)$). Connectedness of $\overline L$ can be shown as in the proof of \cite[Theorem 1.1]{CW2}. 

Lastly, it remains to show (\ref{volume identity}). By the Nash embedding theorem, one can assume that $(M, h)$ is an embedded submanifold of $\mathbb R^{n+K}$ for some natural number $K$. Let $L$ be an immmersed submanifold $M$ and let $\tilde A, \tilde H$ be the second fundamental form and mean curvature vector of $L$ respectively as a submanifold in $\mathbb R^{n+K}$. Then $\tilde A = A + A^M$, where $A^M$ is the second fundamental form of $M$ in $\mathbb R^{n+K}$. Together with the fact that $|A^M| \le C(M)$ since $M$ is compact, we obtain
$$ \left(\int_L |\tilde H|^n d\mu_L \right)^{1/n}\le \left(\int_L |H|^n d\mu_L \right)^{1/n}+ C(M) \mu(L). $$
Thus we can apply \cite[Proposition 4.1]{CW2} to $L = L_k$ and conclude that
\begin{equation} \label{volume upper estimates}
    \operatorname{Vol}( B_\rho (p)\cap L_k) \le C (|\ln \rho| + 1)^n \rho ^n
\end{equation}
for all $p\in M$ and $\rho >0$ small enough, here $C$ depends only on $M$ and $C_V$, $C_A$ in Theorem \ref{main thm}. For any $\rho >0$, by (\ref{volume upper estimates}) we have 
\begin{equation} \label{Volume inequality for L_k, L_k - S_rho}
\operatorname{Vol} (L_k)-C(|\ln \rho| +1)^n \rho^n \le \operatorname{Vol} (L_k \setminus S_\rho) \le \operatorname{Vol} (L_k). \end{equation}
Since $(L_k)$ converges smoothly graphically to $L$ away from $S_\rho$, taking $k\to\infty$ in (\ref{Volume inequality for L_k, L_k - S_rho}) gives 
$$\lim_{k\to\infty} \operatorname{Vol} (L_k)-C(|\ln \rho| +1)^n \rho^n \le \operatorname{Vol} (L\setminus  S_\rho) \le\lim_{k\to \infty} \operatorname{Vol} (L_k).$$
Thus (\ref{volume identity}) is shown by taking $\rho \to 0$ and this finishes the proof of Theorem \ref{main thm}. 

%One can construct, as in \cite[Section 4]{B}, a limit immersion $\iota_\infty : N_\infty \to M\setminus S$ which is HSL. Let $L= \iota_\infty (N_\infty)$ and let $\mu_L$ be the push-forward measure of $\iota_\infty$. 
 \end{proof}

\bibliographystyle{abbrv}

\end{document}